\documentclass[a4paper, 12pt]{amsart}

\usepackage{ucs}
\usepackage[utf8x]{inputenc}
\usepackage[ngerman, english]{babel}
\usepackage{amsfonts,amsmath, amssymb}
\usepackage{fancyhdr}
\usepackage[fixlanguage]{babelbib}
\usepackage{url}
\usepackage[arrow, matrix, curve]{xy}
\usepackage{hyperref}
\usepackage{color}
\usepackage{mathtools, tensor}
\usepackage{cite}
\usepackage{esint}
\usepackage{faktor}

\newtheorem{thm}{Theorem}[section]
\newtheorem{lem}[thm]{Lemma}
\newtheorem{cor}[thm]{Corollary}
\newtheorem{prop}[thm]{Proposition}

\theoremstyle{definition}
\newtheorem{defi}[thm]{Definition}
\newtheorem{ex}[thm]{Example}

\theoremstyle{remark}
\newtheorem{rem}[thm]{Remark}

\numberwithin{equation}{thm}

\newcommand{\eps}{\varepsilon}
\newcommand{\mfdspace}{\mathcal{M}(n,D)}
\newcommand{\mfdspaceCon}{\mathcal{M}(n,D,C)}
\newcommand{\bbN}{\mathbb{N}}
\newcommand{\bbS}{\mathbb{S}}

\newcommand{\dGH}{d_{\text{GH}}}

\DeclareMathOperator{\inj}{inj}
\DeclareMathOperator{\diam}{diam}
\DeclareMathOperator{\vol}{vol}

\begin{document}

\renewcommand{\labelenumi}{\roman{enumi})}

\title[Collapse of codimension 1]{A characterization of codimension one collapse under bounded curvature and diameter}
\author{Saskia Roos}
\address{Max-Planck-Institut für Mathematik, Vivatsgasse 7, 53111 Bonn, Germany}
\email{saroos@mpim-bonn.mpg.de}

\begin{abstract}
Let $\mfdspace$ be the space of closed $n$-dimensional Riemannian manifolds $(M,g)$ with $\diam(M) \leq D$ and $\vert \sec^M \vert \leq 1$. In this paper we consider sequences $(M_i,g_i)$ in $\mfdspace$ converging in the Gromov-Hausdorff topology to a compact metric space $Y$. We show on the one hand that the limit space of this sequence has at most codimension $1$ if there is a positive number $r$ such that the quotient $\frac{\vol(B^{M_i}_r(x))}{\inj^{M_i}(x)}$ can be uniformly bounded from below by a positive constant $C(n,r,Y)$ for all points $x \in M_i$. On the other hand, we show that if the limit space has at most codimension $1$ then for all positive $r$ there is a positive constant $C(n,r,Y)$ bounding the quotient $\frac{\vol(B^{M_i}_r(x))}{\inj^{M_i}(x)}$ uniformly from below for all $x \in M_i$. As a conclusion, we derive a uniform lower bound on the volume and a bound on the essential supremum of the sectional curvature for the closure of the space consisting of all manifolds in $\mfdspace$ with $C \leq \frac{\vol(M)}{\inj(M)}$ .
\end{abstract}

\maketitle

\section{Introduction}

Let $\mfdspace$ be the space of isometry classes of closed $n$-dimensional Riemannian manifolds $(M,g)$ with $\diam(M) \leq D$ and $\vert \sec^M \vert \leq 1$. Due to Gromov, \cite{GromovPreComp}, it is known that this space is precompact with respect to the Gromov-Hausdorff metric $\dGH$. A $\dGH$-convergent sequence $(M_i, g_i)_{i \in \bbN}$  in $\mfdspace$ is said to \textit{collapse} if its limit space is of lower dimension.

The first nontrivial example of  collapse was discovered and carried out by Marcel Berger in about 1962. He considered the Hopf fibration $\bbS^1 \rightarrow \bbS^3 \rightarrow \bbS^2(\frac{1}{2})$ and rescaled the metric tangent to the fibers by $\eps >0$ while keeping the original metric in the directions orthogonal to the fiber fixed. As $\eps \rightarrow 0$ the sectional curvature remains bounded while the injectivity radius converges uniformly to $0$ at each point. Furthermore, $\bbS^3$ resembles more and more a $2$-sphere with constant sectional curvature $4$ as $\eps \rightarrow 0$.

A cornerstone for the theory of collapse under bounded curvature is Gromov's characterization of almost flat manifolds \cite{GromovAlmostFlat}. For example, Fukaya's fibration theorems \cite{FukayaCollapse1}, \cite{FukayaCollapse2} can be understood as a parametrized version of \cite{GromovAlmostFlat}. In \cite{FukayaBoundary}, Fukaya applied these fibration theorems to the sequence of frame bundles of a collapsing sequence in $\mfdspace$ and derived a description of the boundary of $\mfdspace$ (see \cite[Theorem 10.1]{FukayaBoundary}). Cheeger, Fukaya and Gromov proved a simultaneously equivariant and parametrized version of this result  in \cite{CheegerFukayaGromov}. 

It is well-known that in general the elements in the boundary of $\mfdspace$ have singularities. Fukaya showed in \cite{FukayaBoundary} that the Hausdorff dimension of elements in the boundary of $\mfdspace$ is an integer. If the limit space of a collapsing sequence in $\mfdspace$  has codimension $1$ , Fukaya proved that it has to be a Riemannian orbifold with a $C^{1, \alpha}$-metric, \cite[Proposition 11.5]{FukayaOrbifold}. This motivates the main result of this paper which provides the following equivalent characterizations on sequences in $\mfdspace$ which drop at most one dimension in the limit.

\begin{thm}\label{MainTheorem}
Let $(M_i , g_i)_{i \in \bbN}$ be a sequence in $\mfdspace$ converging to a compact metric space $(Y,d)$ in the Gromov-Hausdorff topology. Then the following are equivalent
\begin{enumerate}
	\item $\dim_{\mathrm{Haus}}(Y) \geq (n-1)$,
		\item for all $r >0$ there is a positive constant $C(n,r,Y)$ such that
		\begin{align}\label{condition}
C \leq \frac{\vol(B^{M_i}_r(x))}{\inj^{M_i}(x)}
\end{align} holds for all $x \in M_i$ and all $i \in \bbN$,
	\item for some $r >0$ and a positive constant $C(n,r,Y)$ such that  \eqref{condition} holds for all $x \in M_i$ and all $i \in \bbN$.
\end{enumerate}
\end{thm}

The idea behind Theorem \ref{MainTheorem} is the following illustrative observation. Let $(M_i, g_i)_{i \in \bbN}$ be a collapsing sequence in $\mfdspace$. Then the $r$-balls around a sequence of points $p_i \in M_i$ contain all collapsing directions, while the injectivity radii at the points $p_i$ only represents the fastest scale of collapsing. Now, if the collapse has codimension $1$, it happens on the scale of the injectivity radius. Hence, the volume of the balls $B^{M_i}_r(p_i)$ and the injectivity radii $\inj^{M_i}(p_i)$ converge to $0$ at the same rate. Therefore, the ratio \eqref{condition} can be uniformly bounded from below. However, if the collapse has codimension larger than $2$, then the injectivity radius only represents the fastest scale of collapsing. Thus, the volume of the balls $B^{M_i}_r(p_i)$ converges on a larger scale to $0$ than the injectivity radii $\inj^{M_i}(p_i)$. Consequently, their quotient converges to $0$. 

\begin{ex}
Consider the following sequence of flat tori $(T_i \coloneqq \bbS^1 \times \frac{1}{i} \bbS^1 , g_i)_{i \in \bbN}$. Here $g_i$ is the product metric on on the product of circles of radii $1$ and $\frac{1}{i}$. This sequence collapses to $\bbS^1$ which is of codimension 1. As $\inj^{T_i}(x) \equiv \frac{\pi}{i}$, for any $ r >0$, there exists $I \in \bbN$ such that $\vol(B^{T_i}_r(x)) = 4 \pi^2 \left( \frac{\max\lbrace r, \pi \rbrace}{i} \right)$ for all $x \in T_i$, $i > I$. For simplicity we assume $r \leq \pi$. Combing these, we derive
\begin{align*}
\lim_{i \rightarrow \infty} \frac{\vol(B^{T_i}_r(x))}{\inj^{T_i}(x)} = \lim_{i \rightarrow \infty} \frac{4 \pi^2 \frac{r}{i}}{\frac{\pi}{i}} = 4\pi r > 0.
\end{align*}
Thus, this quotient can be uniformly bounded by a positive constant $C$, as stated in Theorem \ref{MainTheorem}.
\end{ex}

\begin{ex}
Similarly to the previous example we consider the sequence of flat tori $(T_j \coloneqq \frac{1}{j^2} \bbS^1 \times \frac{1}{j} \bbS^1, g_j)_{j \in \bbN}$. This limit collapses to a point. No matter how small we choose $r > 0$ there exists some  $J \in \bbN$ such that $\vol(B^{T_j}_r(x)) = \frac{2 \pi}{j^2} \cdot  \frac{2\pi}{j}$ for any $j > J$. As $\inj^{T_j}(x) \equiv \frac{\pi}{j^2}$, the same considerations as before leads to
\begin{align*} 
\lim_{j \rightarrow \infty} \frac{\vol(B^{T_j}_r(x))}{\inj^{T_j}(x)} = \lim_{j \rightarrow \infty} \frac{4 \pi}{j} =0 .
\end{align*}
Therefore, we cannot find a uniform positive lower bound for this quotient.
\end{ex}

The proof of Theorem \ref{MainTheorem} requires the improved version of Fukaya's fibration theorems derived in \cite{CheegerFukayaGromov}. First, we show that in order to prove Theorem \ref{MainTheorem}, it is enough to restrict to sequences of manifolds with invariant metrics, as introduced in \cite{CheegerFukayaGromov}. Then we proof that i) implies ii) by constructing a lower bound as required in \eqref{condition} for any given $r>0$. As the implication from ii) to iii) is trivial it remains to show that iii) implies i). This direction will be proved by contradiction. We will bound the volume of the ball in the manifold, up to a constant, by the injectivity radius and the diameter of the collapsing fibers. It remains to bound the injectivity of the fibers from above by the injectivity radius of the manifold in the related points. This is done by modifying the results of \cite{Tapp} for bounded Riemannian submersions. In the end, we show that the constructed upper bound on the quotient converges to $0$, giving a contradiction. 

The paper is organized as follows. In Section 2 we recall the theory of collapsing sequences in $\mfdspace$ developed by Cheeger, Fukaya and Gromov. In particular, we recall that each sufficiently collapsed manifold is a singular fibration with infranil fibers. 

Section 3 deals with bounded Riemannian submersions $\eta: M \rightarrow Y$, i.e.\ Riemannian submersions where the norm of the fundamental tensors $A$ and $T$ is bounded by positive constants $C_A$ resp.\ $C_T$. There we modify the results of \cite{Tapp} to obtain the following upper bound on the injectivity radius of the fibers $F_p \coloneqq \eta^{-1}(\lbrace p \rbrace )$, $p \in Y$.

\begin{prop}\label{injFiber}
Let $\eta: M^{n+k} \rightarrow Y^{n}$ be a bounded Riemannian submersion with $\vert \sec^M \vert \leq K$ for some $K > 0$. Assume further that $\inj^M(x) < \min \lbrace \frac{\pi}{\sqrt{K +  3C_A^2}} , \; \frac{1}{4}\inj^Y(p) \rbrace$ for some $x \in F_p$. Then 
\begin{align*}
\inj_{F_p} \leq \Big(1 + \tau(\inj^M(x) \vert C_A, C_T, k, K) \Big) \cdot \inj^M(x).
\end{align*}
Here $\tau(\eps \vert a_1 , \ldots , a_l)$ denotes a positive continuous function such that $\lim_{\eps \rightarrow 0} \tau(\eps \vert a_1 , \ldots , a_l) =0$ for fixed $ a_1 , \ldots , a_l$. The explicit expression of the constant $\Big(1 + \tau(\inj^M(x) \vert C_A, C_T, k, K) \Big)$ is given in the proof.
\end{prop}

In Section 4 we prove Theorem \ref{MainTheorem} using the strategy explained above. In conclusion we define the space $\mfdspaceCon$ consisting of all manifolds in $\mfdspace$ satisfying \eqref{condition} for fixed numbers $r$ and $C$. We show that there is a uniform bound on the essential supremum of the sectional curvature and a uniform lower bound on the volume for all $(n-1)$-dimensional metric spaces $Y$ in the closure of $\mfdspaceCon$.

\subsection*{Acknowledgments}

First I want to especially thank Werner Ballmann and Bernd Ammann for many helpful and enlightening discussions. I am very thankful for the great hospitality of the Max-Planck Institute for mathematics in Bonn. Furthermore, I would like to thank Renato G.\ Bettiol for proof reading and comments improving this paper.


\section{Collapsing theory for bounded curvature and diameter}

In this section we recall the relevant theorems about convergence and collapsing in $\mfdspace$. From this point on, we use the following notation: $\tau(c \vert x_1, \ldots , x_k)$ denotes a positive continuous function depending on $c$ and $x_1, \ldots, x_k$ such that $\lim_{c \rightarrow 0} \tau(c \vert x_1, \ldots , x_k) = 0$ for any fixed $x_1, \ldots , x_k$.

In \cite[Theorem 5.3]{GromovPreComp}, Gromov proved that $\mfdspace$ is precompact in the Gromov-Hausdorff topology. Furthermore, for any fixed positive $i_0$ the subspace 
\begin{align*}
\mathcal{M}(n,D,i_0) \coloneqq \lbrace (M,g) \in \mfdspace \vert \inj^M > i_0 \rbrace
\end{align*}
is precompact in the $C^{1, \alpha}$-topology (see \cite{CheegerFiniteness}, \cite{PetersC1a}). 

Without the uniform lower bound on the injectivity radius a sequence $(M_i, g_i)_{i \in \bbN}$ in $\mfdspace$ might \textit{collapse} to a metric space of lower Hausdorff dimension. In \cite{FukayaBoundary} Fukaya studied collapsing sequences in $\mfdspace$ considering the corresponding sequence of frame bundles $FM_i$.

\begin{thm}[Fukaya]\label{FukayaFiber}
Let $(M_i, g_i)_{i \in \bbN}$ be a sequence in $\mfdspace$ converging with respect to the Gromov-Hausdorff metric to a compact metric space $Y$. There exists a positive $\eps \coloneqq \eps(n,D)$ such that, for all $i$ with $\dGH(M_i,Y) \leq \eps$, there is a map $\eta_i: M_i \rightarrow Y$ and a manifold $\tilde{Y}$ on which $O(n)$ acts isometrically, and an $O(n)$-equivariant map  $\tilde{\eta}_i:FM_i \rightarrow \tilde{Y}$ such that 
\begin{align}\label{diagramm}
	\begin{xy}
	\xymatrix{
	FM_i \ar[r]^{\tilde{\eta}_i} \ar[d]_{\pi_i} & \tilde{Y} \ar[d]^{\pi} \\
	M_i \ar[r]_{\eta_i} & Y 
	} 
	\end{xy} 
\end{align}
commutes, and:
\begin{enumerate}
	\item $\tilde{Y}$ is a Riemannian manifold with  $C^{1, \alpha}$-metric tensors,
	\item $\tilde{\eta}_i$ is a fiber bundle with affine structure group and infranil fibers,
	\item $\tilde{\eta}_i$ is an almost Riemannian submersion, i.e.\ if $X \in T_x FM_i$ is perpendicular to the fibers of $\tilde{\eta}_i$ then 
	\begin{align*}
	e^{- \tau(\dGH(M_i,Y) \vert n,D)} < \frac{\Vert d\tilde{\eta}_i(X) \Vert}{\Vert X \Vert} < e^{\tau(\dGH(M_i,Y) \vert n, D)},
	\end{align*}
	\item $M_i$ and $Y$ are isometric to $FM_i / O(n)$ and $\tilde{Y}/O(n)$ respectively,
	\item for each $p \in Y$ the groups $G_{\tilde{p}} = \lbrace g \in O(n) \vert g(\tilde{p}) = \tilde{p})$ for $\tilde{p} \in \pi^{-1}(p)$ are isomorphic to each other. We set $G_{p} \coloneqq G_{\tilde{p}}$ for some fixed $\tilde{p}\in \pi^{-1}(p)$.
\end{enumerate}
\end{thm}

We henceforth use the notation introduced in this theorem repeatedly.

Another approach to collapse under bounded curvature was carried out by Cheeger and Gromov (\cite{CheegerGromovCollapse1}, \cite{CheegerGromovCollapse2}). They generalized local group actions and introduced an action of a sheaf of groups. In particular, they considered actions of sheaves of tori with additonal regularity conditions. This defines the so-called $F$-structure (``$F$" stands for \textit{flat}). Cheeger and Gromov proved that each sufficiently collapsed complete Riemannian manifold admits an $F$-structure of positive rank. This approach does not require an upper bound on the diameter of the manifold.

Combining these two approaches, Cheeger, Fukaya and Gromov introduced in \cite{CheegerFukayaGromov} a nilpotent structure ($N$-structure) and showed its existence on each sufficiently collapsed part of a complete Riemannian manifold. Roughly, if $M$ is sufficiently collapsed, its frame bundle $FM$ is the total space of a fibration with infranil fibers and affine structural group. Thus, there is a sheaf on $FM$ whose local sections are given by local right invariant vector fields on the fiber. 

A further main result of their article \cite{CheegerFukayaGromov} is the existence of \textit{invariant metrics} on manifolds admitting an $N$-structure. These metrics are invariant in the sense that the local sections of the sheaf on the frame bundle are given by local Killing fields.

To obtain such a metric they first applied the following theorem due to Abresch \cite{Abresch} to obtain uniform bounds on the derivatives of the curvature (see also \cite{Shi}) . 

\begin{thm}[Abresch]\label{SmoothMetric}
For any $\eps >0$ and $n \in \bbN$ there is a smoothing operator $S_{\eps}$ such that on any complete Riemannian manifold $(M,g)$ with $\vert \sec^g \vert \leq 1$ the metric $\tilde{g} \coloneqq S_{\eps}(g)$ satisfies
\begin{enumerate}
	\item $e^{-\eps}g < \tilde{g} < e^{\eps}g$,
	\item $\vert \nabla - \tilde{\nabla} \vert < \eps$,
	\item $\vert \tilde{\nabla}^j \tilde{R} \vert < A_j(n,\eps)$ for all $j \geq 0$.
\end{enumerate}
\end{thm}

Additionally, Rong showed that, for sufficiently small $\eps$, we have the following bounds for the sectional curvature of $ S_{\eps}(g)$, c.f.\ \cite[Proposition 2.5]{RongSectionalCurv}.

\begin{prop}[Rong]\label{RongSec}
There is a constant $\delta > 0$ such that for any complete Riemannian manifold $(M,g)$ with $\vert \sec^g \vert \leq 1$ and any $0 \leq \eps \leq \delta$, there is a positive constant $C(n)$ such that the metric $\tilde{g} \coloneqq S_{\eps}(g)$ satisfies
\begin{align*}
\min \sec^{\tilde{g}} - C(n)\eps \leq \sec^g \leq \max \sec^{\tilde{g}} + C(n) \eps.
\end{align*}
\end{prop}

At this point, we introduce the following notion: A Riemannian manifold $(M,g)$ is called \textit{$A$-regular} if there is a sequence $A =(A_j)_{j \in \bbN}$ of nonnegative numbers such that $\vert \nabla^j R \vert \leq A_j$ for all $j \geq 0$. Furthermore, $C(A)$ will denote a constant depending on $A_j$ for finitely many $j \geq 0$.

Assuming the manifold to be  $A$-regular, Cheeger, Fukaya and Gromov proved the existence of invariant metrics, compare \cite[Section 7, Section 8]{CheegerFukayaGromov}.

\begin{thm}[Cheeger-Fukaya-Gromov]\label{invMetric}
Let $(M,g) \in \mfdspace$  be an $A$-regular Riemannian manifold such that there is a lower dimensional metric space $Y$ with $\dGH(M,Y) \leq \eps(n,D)$, where $\eps(n,D)$ is the constant in Theorem \ref{FukayaFiber}. Then there is an invariant metric $\tilde{g}$ such that
\begin{align*}
\vert \nabla^j(g - \tilde{g}) \vert \leq c(n,A,j) \dGH(M,Y).
\end{align*}
Here $\eps(n,D)$ is the constant from Theorem \ref{FukayaFiber}. Furthermore, the map $\tilde{\eta}: FM \rightarrow \tilde{Y}$ is a Riemannian submersion with respect to the metric induced by $\tilde{g}$ such that the second fundamental form of the fibers is bounded by a positive constant $C(n,A)$.
\end{thm}


\section{The injectivity radius of the fiber}

The goal of this section is to prove Proposition \ref{injFiber}. Therefore, we consider a Riemannian submersion $\eta: M \rightarrow Y$. Henceforth, denotes the fiber over $p \in Y$ by $F_p \coloneqq \eta^{-1}(\lbrace p \rbrace )$  and $k \coloneqq \dim(F_p).$ 

Recall, that a Riemannian submersion is \textit{bounded} if the fundamental tensors $A$ and $T$ are bounded in norm by positive constants $C_A$ resp.\ $C_T$.

The main ingredient of the proof of Proposition \ref{injFiber} is a homotopy with fixed endpoints between a curve $\gamma$ with endpoints in a fiber $F_p$ and a curve $\tilde{\gamma}$ lying completely in the fiber $F_p$ such that the length of $\tilde{\gamma}$ is bounded from above linear in terms of $l(\gamma)$. Such a homotopy was constructed in the proof of Theorem 3.1 in \cite{Tapp}.  

\begin{prop}[Tapp]\label{HomotopCurve}
Let $\eta: M \rightarrow Y$ be a bounded Riemannian submersion with $Y$ being  compact and simply-connected. Then there exists a positive constant $C \coloneqq C(Y,k,C_T, C_A)$ such that any curve $\gamma$ in $M$ with $\eta \circ \gamma$ being a contractible loop is homotopic to a curve $\tilde{\gamma}$ in the fiber $F_p$, $p = \eta \circ \gamma (0)$, satisfying
\begin{align*}
l(\tilde{\gamma}) \leq C \, l(\gamma).
\end{align*}
\end{prop}

Comparing the assumptions of Proposition \ref{HomotopCurve} with those of Proposition \ref{injFiber} there are a few differences. First, in Proposition \ref{HomotopCurve}, Tapp requires $Y$ to be compact and simply-connected. These assumptions are needed to guarantee that for any loop $\alpha: [0,1] \rightarrow Y$ there is is a nullhomotopy $H:[0,1] \times [0,1] \rightarrow Y$, i.e. $H(1,t) = \alpha(t)$, $H(0,t) = \alpha(0)$ and $H(s,0)= H(s,1)= \alpha(0)$ for all $s \in [0,1]$, whose derivatives are uniformly bounded, c.f.\ \cite[Lemma 7.2]{Tapp}. 

Going back to the statement of Proposition \ref{injFiber} the assumptions therein imply that the considered noncontractible geodesic loop $\gamma$ based at $x \in F_p$ has length $l(\gamma) = 2 \inj^M(x) < \frac{1}{2} \inj^Y(p)$. Therefore, the loop $\eta \circ \gamma$ is contractible in $Y$. Furthermore, by assuming a bound on the sectional curvature of $Y$ there is a nullhomotopy for curves with length less or equal $\frac{1}{4}\inj^Y(p))$ whose derivatives can be bounded as follows:

\begin{lem}\label{NullHomotopConst}
Let $Y$ be a Riemannian manifold with $- \lambda^2 \leq  \sec^Y  \leq \Lambda^2$ for some $\lambda, \Lambda > 0$. Furthermore, let $\alpha:[0,1] \rightarrow Y$ be a loop in $Y$ based at $p$ and $l(\alpha) < \min \lbrace \frac{2\pi}{\Lambda} , \frac{1}{2} \inj^Y(p) \rbrace$. Then there is a piecewise smooth nullhomotopy $H: [0,1] \times [0,1] \rightarrow Y$, i.e.\ $H(0,t)= p$ and $H(1,t) = \alpha(t)$ and $H(s,0) = H(s,1)=p$ for all $s \in [0,1]$, such that 
\begin{align*}
\left\vert \frac{\partial}{\partial t} H \right\vert &\leq \frac{\Lambda}{\lambda} \cdot \frac{\sinh(\lambda \; \frac{l(\alpha)}{2})}{\sin(\Lambda \; \frac{l(\alpha)}{2})} \cdot l(\alpha) , \\
\left\vert \frac{\partial}{\partial s} H \right\vert &\leq \frac{\sinh(\lambda \; \frac{l(\alpha)}{2})}{\lambda \frac{l(\alpha)}{2}} \cdot \frac{l(\alpha)}{2}.
\end{align*}
for all $s$, $t \in [0,1]$.
\end{lem}

\begin{proof}
Let $\alpha$ be parametrized proportional to arclength. Since $\alpha$ satisfies $l(\alpha) < \frac{1}{2} \inj^Y(p),$ it lifts to a loop $\tilde{\alpha} \coloneqq \exp_p^{-1} \circ \, \alpha$ in $T_pY$. 

Define $\tilde{H}(s,t) \coloneqq s \cdot \tilde{\alpha}(t)$ with $s$, $t \in [0,1]$. Clearly, we have that 
\begin{align*}
\left\vert \frac{\partial}{\partial s} \tilde{H} \right\vert = \vert \tilde{\alpha} \vert \leq \frac{l(\alpha)}{2}.
\end{align*}  
To estimate $\left\vert \frac{\partial}{\partial t} \tilde{H} \right\vert$ we first observe, that as $- \lambda^2 \leq \sec^Y \leq \Lambda^2$ it follows that 
\begin{align*}
\frac{\sin(\Lambda \vert v \vert)}{\Lambda \vert v \vert} \vert w \vert \leq \vert (\mathrm{D}_{v} \exp_p) (w) \vert \leq \frac{\sinh(\lambda \vert v \vert)}{\lambda \vert v \vert} \vert w \vert
\end{align*}
for all $v \in T_pY$ with $\vert v \vert < \frac{\pi}{\lambda}$ and $w \in T_v T_p Y$, see e.g.\ \cite[Corollary 4.6.1]{Jost}. Therefore, we obtain for $q \in B_{\frac{1}{4} \inj^Y(p)}(p)$ and $u \in T_q Y$ that
\begin{align*}
\vert ( \mathrm{D}_q \exp_p^{-1} ) u \vert \leq \frac{\Lambda \vert \exp^{-1} (q) \vert}{\sin( \Lambda \vert \exp^{-1} (q) \vert )} \vert u \vert.
\end{align*}
Thus,
\begin{align*}
\left\vert \frac{\partial}{\partial t} \tilde{H} \right\vert \leq \left\vert \frac{\partial}{\partial t} \tilde{\alpha} \right\vert \leq  \frac{\Lambda \vert \tilde{\alpha} \vert}{\sin(\Lambda \vert \tilde{\alpha} \vert)} \left\vert \frac{\partial}{\partial t} \alpha \right\vert \leq \frac{\Lambda \frac{l(\alpha)}{2}}{\sin(\Lambda \; \frac{l(\alpha)}{2})} \cdot l(\alpha)
\end{align*}

By construction $H \coloneqq \exp_p (\tilde{H})$ is a piecewise smooth nullhomotopy of $\alpha$ in $Y$ such that
\begin{align*}
\left\vert \frac{\partial}{\partial s} H \right\vert \leq \frac{\sinh(\lambda \vert \tilde{H}(s,t) \vert)}{\lambda \vert \tilde{H}(s,t) \vert} \, \left\vert \frac{\partial}{\partial s} \tilde{H} \right\vert \leq \frac{\sinh(\lambda \; \frac{l(\alpha)}{2} )}{\lambda \frac{l(\alpha)}{2}} \, \cdot \frac{l(\alpha)}{2}.
\end{align*}
The corresponding bound on $\left\vert \frac{\partial}{\partial t} H \right\vert$ is derived similarly.
\end{proof}

The next corollary follows immediately by adjusting the bounds on the derivative of the exponential map.

\begin{cor}
Let $Y$ be a Riemannian manifold with $- \lambda^2 \leq \sec^Y \leq - \Lambda^2$ for some $\lambda \geq \Lambda \geq 0$. Furthermore, let $\alpha : [0,1] \rightarrow Y$ be a loop in $Y$ based at $p$ and $l(\alpha) < \frac{1}{2} \inj^Y(p)$. Then there is a piecewise smooth nullhomotopy $H:[0,1] \times [0,1] \rightarrow Y$, as before, such that
\begin{align*}
\left\vert \frac{\partial}{\partial t} H \right\vert &\leq \frac{\Lambda}{\lambda} \cdot \frac{\sinh(\lambda \; \frac{l(\alpha)}{2})}{\sinh(\Lambda \; \frac{l(\alpha)}{2})} \cdot l(\alpha) , \\
\left\vert \frac{\partial}{\partial s} H \right\vert &\leq \frac{\sinh(\lambda \; \frac{l(\alpha)}{2})}{\lambda \frac{l(\alpha)}{2}} \cdot \frac{l(\alpha)}{2}.
\end{align*}
In the case of $\Lambda=0$, we set $\frac{\sinh(\Lambda)}{\Lambda} = 1$.
\end{cor}

Next, we prove Proposition \ref{injFiber}. Therein we keep carefully track of the dependence of the constants on $\inj^M(x)$ because this is the quantity going to $0$ in a collapsing sequence while the other quantities will be uniformly bounded.

\begin{proof}[Proof of Proposition \ref{injFiber}]
As $\inj^M(x) < \frac{\pi}{\sqrt{K}}$ there is a noncontractible geodesic loop $\gamma$ based at $x$ such that $l(\gamma) = 2 \inj^M(x)$. As $\inj^M(x) < \frac{1}{4} \inj^Y(p)$, the composition $\eta \circ \gamma$ is contractible in $Y$. By Proposition \ref{HomotopCurve}, $\gamma$ is homotopic to a noncontractible loop $\tilde{\gamma}$ in the fiber $F_p$ such that $l( \tilde{\gamma} ) \leq C \cdot l(\gamma)$ for a positive constant $C \coloneqq C(Y,k,C_A,C_T)$. Thus,
\begin{align*}
2 \inj^{F_p} \leq l(\tilde{\gamma}) \leq C \cdot l(\gamma) = C \cdot 2 \inj^M(x).
\end{align*} 

We claim that $C = \tau( l(\gamma) \vert C_A, C_T, k, K )$. The proof consists of a careful study of the constant $C$, following the proof of \cite[Theorem 3.1]{Tapp}. In this proof, Tapp modifies the path $\gamma$ such that it is a concatenation of paths with endpoints in the fiber whose length is not larger than $( 2 \diam (Y) + 1)$. As $l(\gamma) < (2 \diam (Y) + 1)$ already holds by assumption we do not need this modification.

Set $\alpha(t) \coloneqq \eta \circ \gamma$. The vertical curve $\tilde{\gamma}$ is the concatenation of the paths $\beta_1$ and $\beta_2$. The first path $\beta_1$ goes from $x$ to $z \coloneqq h^{\alpha}(x)$, where $h^{\alpha}: F_p \rightarrow F_p$ is the holonomy diffeomorphism associated to $\alpha$ and the path $\beta_2$ connects $z$ with $x$ again. Hence, 
\begin{align*}
l(\tilde{\gamma}) \leq l(\beta_1) + l(\beta_2) \leq P \cdot l(\gamma) + L \cdot l(\gamma) = C l(\gamma)
\end{align*}
for some explicit positive constante $P$ and $L$, compare with \cite[p.\ 645]{Tapp}. We will study these constants $P$ and $L$ in detail.

First we consider the inequality $l(\beta_1) \leq P \cdot l(\gamma)$. The constant $P$ is a bound on the derivative of the function $l \mapsto \rho_l (1)$ between $l = 0$ and $l = l(\gamma)$, where $\rho_{l(\gamma)} (t)$ is the solution to the differential equation
\begin{align}\label{ODE}
(\rho_{l(\gamma)})^{\prime}(t) = kC_A  Q_s Q_t l(\gamma) ( 1 + 4^k k!) + k Q_t l ( C_T + 4^k k! C_A) \rho_l(t) \ ; \ \rho_{l(\gamma)}(0) = 0
\end{align}
The constants $Q_t$ and $Q_s$ are the bounds on the nullhomotopy $H$ of the path $\alpha(t)$ in $Y$, i.e.\  $\left\vert \frac{\partial}{\partial t} H \right\vert \leq Q_t l(\gamma)$ and $\left\vert \frac{\partial}{\partial s} H \right\vert \leq Q_s$. uniformly

As, by Gray-O'Neill's formula, $-K \leq \sec^Y \leq ( K + 3C_A^2 )$ and, by assumption, $l(\alpha) \leq l(\gamma) < \min \lbrace \frac{2 \pi}{\sqrt{K + 3C_A^2}}, \frac{1}{2} \inj^Y(p)\rbrace$ we apply Lemma \ref{NullHomotopConst} and obtain
\begin{align*}
Q_t &= \frac{\sqrt{K + 3C_A^2}}{\sqrt{K}} \cdot \frac{\sinh( \sqrt{K} \frac{l(\gamma)}{2} )}{\sin(\sqrt{K + 3 C_A^2} \frac{l(\gamma)}{2})} , \\
Q_s &= \frac{\sinh(\sqrt{K} \frac{l(\gamma)}{2})}{\sqrt{K} \frac{l(\gamma)}{2}} \cdot \frac{l(\gamma)}{2} \eqqcolon \tilde{Q}_s l(\gamma).
\end{align*}
Note that for any loop $\bar{\alpha}$ in $Y$ of length less or equal than $l(\gamma)$ the corresponding nullhomotopy $\bar{H}$ of $\bar{\alpha}$ satisifies the bounds $\left\vert \frac{\partial}{\partial t} \bar{H} \right\vert \leq Q_t l(\bar{\alpha})$ and $\left\vert \frac{\partial}{\partial s} \bar{H} \right\vert \leq \tilde{Q}_s l(\bar{\alpha})$

Thus, in our case the differential equation \eqref{ODE} reads as
\begin{align*}
(\rho_l)^{\prime}(t) = k C_A \tilde{Q}_s Q_t l^2 (1 + 4^k k!) + k Q_t l (C_T + 4^k k! C_A) \rho_l (t) \ ; \ \rho_l(0) = 0,
\end{align*}
for $0 \leq l \leq l(\gamma)$, compare \cite[Lemma 3.3]{Tapp}. For simplicity, set
\begin{align*}
G_1 &\coloneqq k C_A \tilde{Q}_s Q_t (1 + 4^k k!), \\
G_2 &\coloneqq k Q_t (C_T + 4^k k! C_A).
\end{align*}  
Using the variation of constants, we conclude
\begin{align*}
\rho_l(1) = \frac{G_1}{G_2} l \big( e^{G_2 l} -1 \big).
\end{align*}
Therefore,
\begin{align}\label{BoundtheODE}
\frac{\mathrm{d}}{\mathrm{d} l} \rho_l(1) = \frac{G_1}{G_2}\big( e^{G_2 l} -1 \big) + G_1 l e^{G_2 l} \leq  \frac{G_1}{G_2}\big( e^{G_2 l(\gamma)} -1 \big) + G_1 l(\gamma) e^{G_2 l(\gamma)} \eqqcolon P.
\end{align}
It remains to check the behavior of the appearing quantities as $l(\gamma)$ becomes small. As $Q_t$ and $\tilde{Q}_s$ are the only appearing quantities that depend on $l(\gamma)$ we first note that
\begin{align*}
\lim_{l(\gamma) \rightarrow 0 } Q_t &=  1\\ 
\lim_{l(\gamma) \rightarrow 0} \tilde{Q}_s &= \frac{1}{2}.
\end{align*}
Therefore, we extract the quantities $Q_t$ and $\tilde{Q}_s$ in \eqref{BoundtheODE}. This is done by considering each term separately. We observe that
\begin{align*}
\frac{G_1}{G_2} &= \frac{k C_A \tilde{Q}_s Q_t (1 + 4^k k!)}{k Q_t (C_T + 4^k k! C_A)} = \tilde{Q}_s \frac{C_A(1+4^k k!)}{C_T + 4^k k! C_A} \eqqcolon \tilde{Q}_s \cdot C_1(C_A,C_T, k), \\
G_1 l(\gamma) &= k C_A \tilde{Q}_s Q_t (1 + 4^k k!) l(\gamma) \eqqcolon \tilde{Q}_s Q_t l(\gamma) \cdot C_2 (C_A, k),\\
G_2 l(\gamma) &= k Q_t (C_T + 4^k k! C_A) l(\gamma) \eqqcolon Q_t l(\gamma) \cdot C_3(C_A,C_T,k).
\end{align*}
As $l(\gamma)$ becomes small, we obtain
\begin{align*}
\lim_{l(\gamma) \rightarrow 0} \frac{G_1}{G_2} &= \frac{1}{2} \cdot C_1(C_A, C_T, k),\\
\lim_{l(\gamma) \rightarrow 0} G_1 l(\gamma) &= 0 \cdot C_2(C_A, k) = 0,\\
\lim_{l(\gamma) \rightarrow 0} G_2 l(\gamma) &= 0 \cdot C_3(C_A,C_T,k) = 0.
\end{align*}
Summarizing these observations we conclude
\begin{align*}
\lim_{l(\gamma)\rightarrow 0} P &= \lim_{l(\gamma)\rightarrow 0} \frac{G_1}{G_2}\big( e^{G_2 l} -1 \big) + G_1 l e^{G_2 l} \leq \lim_{l(\gamma)\rightarrow 0} \frac{G_1}{G_2}\big( e^{G_2 l(\gamma)} -1 \big) + G_1 l(\gamma) e^{G_2 l(\gamma)}\\
&= \frac{1}{2}C_1(C_A,C_T,k) \big(e^0 -1 \big) + 0 \cdot e^0 = 0.
\end{align*}
This shows that $P = \tau(l(\gamma) \vert C_A, C_T, k, K)$.

Next, we consider the inequality $l(\beta_2) \leq  L \cdot l(\gamma)$. Here, $L$ is the maximum of the Lipschitz constants of the holonomy diffeomorphism $h^{\alpha}$ associated to paths $\alpha$ in $Y$. Since $h^{\alpha}$ satisfies the Lipschitz constant $e^{C_T \cdot l(\alpha)}$ (c.f.\ \cite[Lemma 4.2]{HolonomyDiffeo}) and $l(\alpha)$ is bounded from above by $l(\gamma)$ we conclude that
\begin{align*}
L = e^{C_T \cdot l(\gamma)} =1 + \tau(l(\gamma) \vert C_T ).
\end{align*}
Together with $P = \tau(l(\gamma) \vert C_A, C_T, k, K)$, this shows the claim.

Thus, we finally obtain
\begin{align*}
2\inj^{F_p} \leq l(\tilde{\gamma}) &\leq C \cdot l(\gamma) \\
&= \Big(1 + \tau ( l(\gamma) \vert C_A, C_T, k, K) \Big) \cdot l(\gamma) \\
&= \Big(1 + \tilde{\tau}( \inj^M(x) \vert C_A, C_T, k, K) \Big) \cdot 2 \inj^M(x).
\end{align*}
\end{proof}

\begin{rem}
If $- K \leq \sec^M \leq - \kappa$ for some $\kappa >0$ such that $(-\kappa + 3 C_A^2) \leq 0$ then the assumpion $\inj^M(x) < \frac{1}{4}\inj^Y(p)$ is already sufficient for Proposition \ref{injFiber} to hold, as $M$, as well as $Y$, do not have any conjugate points. In particular, the injectivity radius at some point at $M$ equals half of the length of the shortest geodesic loop based at that point.
\end{rem}


\section{Characterization of Codimension One Collapse}

In this section, we prove Theorem \ref{MainTheorem} and discuss the properties of the subspace of $\mfdspace$ consisting of manifolds satisfying the condition \eqref{condition} for chosen $r$ and $C$.

We first observe that, in the case of a noncollapsing sequence, the statement of Theorem \ref{MainTheorem} is obviously true, as the limit space is a closed Riemannian manifold of the same dimension. Thus, we only consider the case of collapsing sequences in $\mfdspace$.

The strategy of the proof of Theorem \ref{MainTheorem} is to first reduce the statement to sequences of sufficiently collapsed manifolds with invariant metrics in the sense of \cite{CheegerFukayaGromov}. Then, we prove Theorem \ref{MainTheorem} in that special case. 

Thus, we first show that, for uniformlya collapsing sequence $(M_i, g_i)_{i \in \bbN}$, we can switch to invariant metrics $\tilde{g}_i$ without affecting the statement of Theorem \ref{MainTheorem}.

\begin{lem}\label{SimplifyMain}
Let $(M_i,g_i)_{i \in \bbN}$ be a collapsing sequence in $\mfdspace$ with limit space $Y$. For $\delta >0$ sufficiently small and all $i \in \bbN$ sufficiently large, there is an invariant metric $\tilde{g}_i$ such that 
\begin{align*}
\vert g_i - \tilde{g}_i \vert < (e^{\delta}-1) + C(n,\delta) \dGH(M_i, Y), \\
\vert \nabla_i - \tilde{\nabla}_i \vert \leq \delta + C_1(n, \delta) \dGH(M_i, Y),\\
\vert \tilde{\nabla}^j_i \tilde{R}_i \vert \leq C(j, n, \delta)(1 + \dGH(M_i,Y)).
\end{align*}
In particular, 
\begin{align*}
e^{- \tau(\dGH(M_i, Y) \vert n, \delta)- \tau(\delta \vert n)} \frac{\vol(\tilde{B}^{M_i}_r(x))}{\widetilde{\inj}^{M_i}(x)} &\leq \frac{\vol(B_r^{M_i})(x)}{\inj^{M_i}(x)}\\
& \leq e^{\tau(\dGH(M_i, Y) \vert n, \delta) + \tau(\delta \vert n)} \frac{\vol(\tilde{B}^{M_i}_r(x))}{\widetilde{\inj}^{M_i}(x)},
\end{align*}
where $\tilde{B}^{M_i}_r(x)$ and $\widetilde{\inj}^{M_i}(x)$ are taken with respect to the metric $\tilde{g}_i$.
\end{lem}

\begin{proof}
First, we apply Theorem \ref{SmoothMetric} with $\delta$ to the sequence and obtain the sequence $(M_i, \hat{g}_i)_{i \in \bbN}$ which consists only of $A$-regular manifolds, with $(A_j(n, \delta))_{j \in \bbN}$. Furthermore, by choosing $\delta$ sufficiently small, Proposition \ref{RongSec} implies that $\vert \widehat{\sec}^{M_i} \vert \leq (1 + c(n) \delta)$. 

It follows, by the estimates for the metrics $g_i$ and $\hat{g}_i$ in Theorem \ref{SmoothMetric} that
\begin{align}\label{change1}
e^{- \tau(\delta \vert n)} \frac{\vol(\hat{B}^{M_i}_r(x))}{\widehat{\inj}^{M_i}(x)} \leq \frac{\vol(B_r^{M_i})(x)}{\inj^{M_i}(x)} \leq e^{\tau(\delta \vert n)} \frac{\vol(\hat{B}^{M_i}_r(x))}{\widehat{\inj}^{M_i}(x)}
\end{align} 
holds for sufficiently large $i \in \bbN$. Here, sufficiently large means that $\inj^{M_i}(x)$, resp.\ $\widehat{\inj}^{M_i}(x)$ is smaller than the conjugate radius of $(M_i, g_i)$, resp.\ $(M_i, \hat{g}_i)$ which is uniformly bounded in terms of the upper sectional curvature bound. The bound on the conjugate radius for $(M_i, \hat{g}_i)$ only changes slightly by choosing  $\delta >0 $ sufficiently small.

Next, we apply Theorem \ref{invMetric} to each element of $(M_i, \hat{g}_i)_{i \in \bbN}$ which satisfies $\dGH(M_i, Y) \leq \eps(n,D)$ to obtain an invariant metric $\tilde{g}_i$. Recall, that $\eps(n,D)$ is the constant from Theorem \ref{FukayaFiber}. This leads to a new sequence $(M_i, \tilde{g}_i)_{i \in \bbN}$. The claimed bounds on $\tilde{g_i}$ follows by combining the inequalities given in Theorem \ref{SmoothMetric} and Theorem \ref{invMetric}. In particular, after a small rescaling, $(M_i, \tilde{g}_i)_{i \in \bbN}$ lies again in $\mfdspace$. 

Furthermore, as $\vert \hat{g}_i - \tilde{g}_i \vert_{C^{\infty}} \leq \tau(\dGH(M_i,Y) \vert n, \delta)$ it follows that
\begin{align}\label{change2}
e^{- \tau(\dGH(M_i, Y) \vert n, \delta)} \frac{\vol(\tilde{B}^{M_i}_r(x))}{\widetilde{\inj}^{M_i}(x)} \leq \frac{\vol(\hat{B}_r^{M_i})(x)}{\widehat{\inj}^{M_i}(x)} \leq e^{\tau(\dGH(M_i, Y) \vert n, \delta)} \frac{\vol(\tilde{B}^{M_i}_r(x))}{\widetilde{\inj}^{M_i}(x)},
\end{align}uniformly
for $i$ sufficiently large, as before.

Observe, that the sequences $(M_i, \hat{g}_i)_{i \in \bbN}$ and $(M_i, \tilde{g}_i)_{i \in \bbN}$ converge to the same limit space $\hat{Y}$. Furthermore, as $\dGH((M_i, g_i), (M_i, \tilde{g}_i) )$ is small, it follows by \cite[Lemma 2.3]{FukayaBoundary} that the Lipschitz-distance between $Y$ and $\hat{Y}$ is also small. In particular, $Y$ and $\hat{Y}$ are homeomorphic and thus have the same Hausdorff dimension. Together with \eqref{change1} and \eqref{change2} the claim follows.
\end{proof}

In order to prove Theorem \ref{MainTheorem} we consider the following \textit{simplified setting}:

Let $(M_i, g_i)_{i \in \bbN}$ be a sequence in $\mfdspace$ converging to a compact metric space $Y$ of lower dimension. There is a large index $I$ such that $\dGH(M_i, g_i) \leq \eps(n,D)$, where $\eps(n,D)$ is the constant from Theorem \ref{FukayaFiber}, and $\inj^{M_i}(x) < \pi$ for all $x \in M_i$, $i \geq I$ . In order to prove Theorem \ref{MainTheorem}, it is  sufficient to consider such sequences $(M_i, g_i)_{i \geq I}$, where we can assume without loss of generality that $g_i$ is an invariant metric such that $(M_i, g_i)$ is $A(n,D,\delta)$-regular for all $i$, by Lemma \ref{SimplifyMain}. Here, we used that $\dGH(M_i,Y)$ is bounded from above by $\eps(n,D)$, for all $i$.
\newline

The next proposition together with Lemma \ref{SimplifyMain} proves the implication i) to ii) in Theorem \ref{MainTheorem}.

\begin{prop}\label{OneTwo}
Let $(M_i, g_i)_{i \in \bbN}$ be a collapsing sequence of $A$-regular manifolds in $\mfdspace$ converging to a compact metric space $Y$ in the Gromov-Hausdorff topology. Suppose that for each $i$, $\dGH(M_i,Y) \leq \eps(n,D)$, and $\inj^{M_i}(x) < \pi$ for all $x \in M_i$, and that the metric $g_i$ is invariant for the corresponding $N$-structure on $M_i .$ 

If $ \dim_{\mathrm{Haus}}(Y) = (n-1)$, then for each $r>0$ there is a positive constant $C \coloneqq C(n,r,Y)$ such that
\begin{align}\label{conditionAgain}
C \leq \frac{\vol(B_r^{M_i}(x))}{\inj^{M_i}(x)}
\end{align}
for all $x \in M_i$ and $i \in \bbN.$
\end{prop}

\begin{proof}
As $\dim_{\mathrm{Haus}}(Y)= (n-1)$ it follows by \cite[Proposition 11.5]{FukayaOrbifold} that $Y$ is a compact Riemannian orbifold. As the fibration, $\tilde{\eta}_i: FM_i \rightarrow \tilde{Y}$ is an $\bbS^1$-bundle, it descends to an $\bbS^1$-orbifold bundle $\eta_i : M_i \rightarrow Y$. 

Fix some $r > 0$. As $i \rightarrow \infty$, the ball $B^{M_i}_r(x)$ more and more resembles $\eta_i^{-1}(B^Y_r(\eta_i(x)))$. Therefore, it follows that there is an index $I$ such that for any $i > I$
\begin{align*}
\eta_i^{-1}(B^Y_{\frac{r}{2}}(p) ) \subset B^{M_i}_r(x)
\end{align*}
holds for all $p \in Y$ and $x \in F^i_p \coloneqq \eta^{-1}_i (p)$ ( e.g.\ one can use Toponogov's theorem as the sequence lies in $\mfdspace$).

Since the $T$-tensor of the Riemannian submersions $\tilde{\eta}_i$ is uniformly bounded by a constant $C_T(n,A)$, see Theorem \ref{invMetric}, it follows, by considering the commuting diagramm \eqref{diagramm}, that for any $r > 0$ there is a positive constant $C_1 \coloneqq C_1(r, n , C_T )$ such that, for all $i > I$,
\begin{align*}
\vol(B^{M_i}_r(x)) \geq  C_1 \vol(B^Y_{\frac{r}{2}}(p)) \vol(F^i_p) =  C_1 \vol(B^Y_{\frac{r}{2}}(p)) \, 2 \inj(F_p).
\end{align*}
The last equality holds as $F^i_p \cong \bbS^1$. Therefore,
\begin{align*}
\frac{\vol(B^{M_i}_{r}(x))}{\inj^{M_i}(x)} & \geq C_1 \vol(B^Y_{\frac{r}{2}}(p))\\
& \geq 2 C_1 \inf_{i \in \bbN} \min_{p \in Y} \vol(B^Y_{\frac{r}{2}}(p)).
\end{align*}

As $Y$ is a smooth compact Riemannian orbifold, the above constant $C$ is strictly positive.
\end{proof}

To finish the proof of Theorem \ref{MainTheorem} it remains to show that iii) implies i). We again assume, without loss of generality, the same simplified setting as explained above. The main idea of the proof is to derive a contradiction by constructing an upper bound converging to $0$. This is done in the next proposition which, toghether  with Lemma \ref{SimplifyMain}, finishes the proof of Theorem \ref{MainTheorem}.

\begin{prop}\label{ThreeOne}
Let $(M_i, g_i)_{i \in \bbN}$ be a collapsing sequence of $A$-regular manifolds in $\mfdspace$ converging to a compact metric space $Y$ in the Gromov-Hausdorff topology. Suppose that for each $i$, $\dGH(M_i,Y) \leq \eps(n,D)$, $\inj^{M_i}(x) < \pi$ for all $x \in M_i$ and that the metric $g_i$ is invariant for the corresponding $N$-structure on $M_i .$ 

Then $\dim_{\mathrm{Haus}}(Y) = (n-1)$ if there is some $r >0$ and $C>0$ such that
\begin{align*}
C \leq \frac{\vol(B_r^{M_i}(x))}{\inj^{M_i}(x)}
\end{align*}
for all $x \in M$ and all $i \in \bbN$.
\end{prop}

\begin{proof}
Let $(M_i, g_i)_{i \in \bbN}$ be a collapsing sequence in $\mfdspace$ such that the Gromov-Hausdorff limit $Y$ has codimension $k \geq 2$. Furthermore, assume that there are positive numbers $r$ and $C$ such that \eqref{conditionAgain} is satisfied for all points $x \in M_i$ and all $i$.

Naber and Tian proved in \cite[Theorem 1.1]{NaberTian} that the limit space $Y$ is a Riemannian orbifold outside a set $S$ of Hausdorff dimension $\dim_{\mathrm{Haus}}(S) \leq \min \lbrace n-5, \dim Y - 3 \rbrace$. Hence, $\hat{Y} \coloneqq Y \smallsetminus S$ is a Riemannian orbifold. 

By Theorem \ref{invMetric} it follows that the fundamental form of $\tilde{\eta}_i: FM_i \rightarrow \tilde{Y}$ is uniformly bounded by a constant $\tilde{C}_T(n,A)$. Therefore, considering the commutive diagramm \eqref{diagramm}, it follows that for any $r > 0$ there is a constant $C_1(r,n,\tilde{C}_T)$ such that
\begin{align*}
\vol(B^{M_i}_r(x)) \leq C_1 \vol(B^Y_{r}(\eta_i(x))) \vol(F^i_{\eta_i(x)})
\end{align*}
for any $x \in M_i$, $i \in \bbN$.

Let $p \in \hat{Y}$ be a regular point and $x \in F^i_p $. Thus, there is some $\kappa >0$ such that $\overline{B^Y_{\kappa}(p)}$ is a compact Riemannian manifold with boundary. 

Now, we consider the maps $\eta_i$ restricted to the preimage  $\eta_i^{-1}\big( \overline{B^Y_{\kappa}(p)} \big)$ for all $i$. These are Riemannian submersions between manifolds. Recall that the $T$-tensor of the Riemannian submersions $\tilde{\eta}_i$ is uniformly bounded in $i$, by Theorem \ref{invMetric}. Hence, also the $T$-tensor of $\eta_i$ restricted to the preimage of $\overline{B^Y_{\kappa}(p)}$ is uniformly bounded by a constant $C_T$.

As the sequence $(M_i,g_i)_{i \in \bbN}$ only consists of $A$-regular manifolds, we can extract a subsequence, denoted by $(M_i, g_i)_{i \in \bbN}$ such that the Riemannian metrics $(\tilde{\eta}_i)_{\ast}(g_i^F)$ on $\tilde{Y}$ converges in $C^{\infty}$. Here $g_i^F$ denotes the metric on the frame bundle $FM_i$ induced by $g_i$. Therefore, it follows that the metrics $(\eta_i)_{\ast}(g_i)$ converges in $C^{\infty}$ on $\overline{B^Y_{\kappa}(p)}$. In particular, the sectional curvature on $B^Y_{\kappa}(p)$ can be uniformly bounded in $i$. Hence, by Gray-O'Neill's formula, the $A$-tensor is also, on $B^Y_{\kappa}(p)$, uniformly bounded in norm by a constant $C_A$. 

Now, let $\gamma$ be the noncontractible geodesic loop based at $x \in F_p$ such that $l(\gamma) = 2 \inj^{M_i}(x)$. By taking $i$ sufficiently large, the assumptions of Proposition \ref{injFiber} are fulfilled and therefore
\begin{align}\label{injEstimate}
\inj^{F^i_p} \leq (1 + \tau(\inj^{M_i}(x) \vert k, C_T, C_A) \, ) \inj_{M_i}(x) \eqqcolon C_2 \inj^{M_i}(x).
\end{align}

It follows easily from Gray-O'Neill's formula that the sectional curvature of $F_p^i$ is bounded from below by $-K^2$ for some positive constant $K \coloneqq K(C_T)$. Hence, we apply \cite[Corollary 2.3.2]{HeintzeKarcher} and \eqref{injEstimate} to conclude
\begin{align*}
C & \leq \frac{\vol(B^{M_i}_r(x))}{\inj^{M_i}(x)} \\
& \leq \frac{C_1 \vol(B^Y_r(p)) \vol(F^i_p)}{\inj^{M_i}(x)} \\
& \leq \frac{C_1 \vol(B^Y_r(p))\bigg( C_3(k) \inj^{F^i_p} \left( \frac{\sinh(\diam(F^i_p) K )}{K} \right)^{k-1} \bigg) }{\inj^{M_i}(x)} \\
& \leq C_1 C_2 C_3  \vol(B^Y_r(p)) \left( \frac{\sinh(\diam(F^i_p) K )}{K} \right)^{k-1} 
\end{align*}

Since $(M_i, g_i)_{i \in \bbN}$ is a collapsing sequence  $\lim_{i \rightarrow \infty} \diam(F^i_p) =0$. As $k \geq 2$  by assumption, it follows that
\begin{align*}
\lim_{i \rightarrow \infty} \left( \frac{\sinh(\diam(F^i_p) K )}{K} \right)^{k-1}  = 0.
\end{align*}
This implies in the limit $ C \leq 0$ which is a contradiction.
\end{proof}

As a further explicit example we consider the Hopf fibration $\bbS^1 \rightarrow \bbS^3 \rightarrow \bbS^2(\frac{1}{2})$. We compute the constant from Theorem \ref{MainTheorem} for this collapsing sequence explicitely.

\begin{ex}[Hopf fibration]
Consider the sequence $(\bbS^3_i, g_i)_{i \in \bbN} \in \mathcal{M}(4, \pi)$. Here $\bbS^3_i$ denotes the total space of the Hopf fibrations where the fibers are scaled such that they have diameter $\frac{\pi}{i}$. This is a collapsing sequence converging to $\bbS^2(\frac{1}{2}) $. The Riemannian submersions $\eta_i: \bbS^3 \rightarrow \bbS^2(\frac{1}{2})$ are all totally geodesic and the integrability tensors are uniformly bounded by $2$. Let $r = \pi$, then 
\begin{align*}
\vol(\bbS^3_i) = \vol(B_{\pi}^{\bbS^3_i}(x)) = \vol(F_p^i) \vol(B_{\frac{\pi}{2}}(p)) = \frac{2 \pi}{i} \vol\big( \bbS^2 \left( 1/2\right) \big) = \frac{2 \pi^2}{i}.
\end{align*}
Thus, we have
\begin{align*}
\frac{\vol(B_{\pi}^{\bbS^3_i}(x))}{\inj^{\bbS^3_i}(x)} = \frac{\frac{2 \pi^2}{i}}{\frac{\pi}{i}} = 2 \pi = 2 \vol(\bbS^2 \left( 1/2 \right) ),
\end{align*}
which bounds this quotient uniformly in $i$.
\end{ex}

We conclude this paper, by examining the subspace of $\mfdspace$ consisting of manifolds satisfying \eqref{condition} for fixed $r$ and $C$.

\begin{defi}
For given positive numbers $n$, $D$, and $C$, we define $\mfdspaceCon$ as the space consisting of all isometry classes of closed Riemannian manifold $(M,g)$ in $\mfdspace$ satisfying
\begin{align*}
C \leq \frac{\vol(M)}{\inj(M)}.
\end{align*}
\end{defi}

By Theorem \ref{MainTheorem} the closure $\mathcal{C}\mfdspaceCon$ of $\mfdspaceCon$ with respect to the Gromov-Hausdorff distance only consists of Riemannian manifolds or Riemannian orbifolds. For simplicity we consider each limit of a sequence in $\mfdspaceCon$ as an orbifold and understand a manifold as a special case. Recall that a manifold is an orbifold where each point is regular.

At this point we want to recall the following proposition due to Fukaya, (c.f.\ \cite[Proposition 11.5]{FukayaOrbifold}:
\begin{prop}[Fukaya]\label{OrbiProp}
Let $(M_i, g_i)_{i \in \bbN}$ be a sequence in $\mfdspace$ converging to a compact metric space $Y$ of codimension $1$, then the groups $G_p$ (defined in Theorem \ref{FukayaFiber}) are all finite. In other words, $Y$ is a Riemannian orbifold.
\end{prop}

In particular, it follows by the definition of a Riemannian orbifold that $\mathcal{C} \mfdspaceCon$ only consists of \textit{smooth} elements. For the definition of smooth elements in the closure of $\mfdspace$, we refer \cite[Definition 0.4]{FukayaBoundary} for the definition.  This observation is important, as we want to use the following lemma due to Fukaya (c.f.\ \cite[Lemma 7.8]{FukayaBoundary}).
\begin{lem}\label{FukayaSec}
Let $(M_i, x_i)_{i \in \bbN}$ be a sequence of pointed manifolds in the $\dGH$-closure of $\mfdspace$ converging to a smooth element $(Y,p)$. Suppose that the sectional curvature of $M_i$ at $x_i$ are unbounded. Then the dimension of the group $G_p$ is positive.
\end{lem}

Combining this with Proposition \ref{OrbiProp} we derive the following properties of $\mathcal{C}\mfdspaceCon$.

\begin{thm}
The space $\mfdspaceCon$ is precompact in the Gromov-Hausdorff topology. Thus, any sequence $(M_i, g_i)_{i \in \bbN}$ contains a subsequence that either converges to a Riemannian manifold of the same dimension in the $C^{1, \alpha}$-topology or  collapses to a compact Riemannian orbifold $Y$ of Hausdorff-dimension $(n-1)$ such that the metrics $(\eta_i)_{\ast} g_i$ converges in the $C^{1, \alpha}$-topology on $Y$. Furthermore, any element $Y$ in  $\mathcal{C}\mfdspaceCon$  with $\dim(Y) = n-1$ satisfies $\Vert \sec^Y \Vert_{L^{\infty}} \leq K$ and $\vol(Y) > v$ for positive constants $v$ and $K$ depending on $n$, $D$ and $C$.
\end{thm}

\begin{proof}
Let $(M_i, g_i)_{i \in \bbN}$ be a sequence in $\mfdspaceCon$. Then there exists by Gromov's compactness result a $\dGH$-convergent subsequence, converging to $Y$. 

If $\dim(Y) = n$ then the injectivity radius of the manifolds $M_i$ is uniformly bounded from below by a constant $i_0$. Thus, this sequence lies in $\mathcal{M}(n,D,i_0)$ and the claim follows as this space is precompact in the $C^{1, \alpha}$-topology.

If $\dim(Y) < n$, then $\dim(Y) = (n-1)$ by Theorem \ref{MainTheorem}. Thus, $Y$ is a Riemannian orbifold by Proposition \ref{OrbiProp}. Applying the $O(n)$-equivariant version of Gromov's compactness result there is a subsequence $(M_i,g_i)_{i \in \bbN}$ such that $(\tilde{\eta}_i)_{\ast}g^F_i$ converges on $\tilde{Y}$ in the $C^{1, \alpha}$-topology to an $O(n)$-equivariant metric. Here $g^F_i$ denotes the metric on $FM_i$ induced by $g_i$. As $Y$ is a Riemannian orbifold, we also obtain that $(\eta_i)_{\ast} g_i$ converges in $C^{1, \alpha}$. This proves the first part of the proposition.

For the second part, assume that there is a sequence $(Y_i)_{i \in \bbN}$ of $(n-1)$-dimensional orbifolds in $\mathcal{C}\mfdspaceCon$ such that there is a sequence of points $p_i \in Y_i$ where the sectional curvatures are unbounded in $i$. As each element $Y_i$ can be reached by a collapsing sequence in $\mfdspace$, there is a subsequence $(Y_i)_{i \in \bbN}$ converging to an element $Y_{\infty}$ in $\mathcal{C}\mfdspaceCon$ and a point $p_{\infty}$ with unbounded sectional curvature. By a diagonal sequence argument and Theorem \ref{MainTheorem} it follows that $Y_{\infty}$ is a Riemannian orbifold. As $\mathcal{C}\mfdspaceCon$ is a subset of the $\dGH$-closure of $\mfdspace$ we can apply Lemma \ref{FukayaSec}. It follows that the group $G_{p_{\infty}}$ has positive dimension. This is a contradiction, as the group $G_{p_{\infty}}$ is finite finite by Proposition \ref{OrbiProp}. Thus, there exists a constant $K(n,D,C,r)$ as claimed. 

The volume bound follows also by contradiction. Hence, there exists a sequence $(Y_i, h_i)_{i \in \bbN}$ such that $\dim(Y_i) = (n-1)$ for all $i$ and $\vol(Y_i)$ converging to $0$ as $i$ tends to infinity. In other words the sequence $(Y_i, h_i)_{i \in \bbN}$ collapses. By a diagonal sequence argument one obtains a sequence $(M_j, g_j)_{j \in \bbN}$ in $\mfdspaceCon$ converging to a space of at least codimension $2$. This is a contradiction, by Theorem \ref{MainTheorem}.
\end{proof}


\nocite{*}
\bibliography{literature.bib}
\bibliographystyle{amsalpha}

\end{document}